\documentclass[12pt,a4paper,oneside]{amsart}

\usepackage{amsfonts, amsmath, amssymb, amsthm, amscd, hyperref}

\usepackage{graphicx}

\usepackage{anysize}
\marginsize{2cm}{2cm}{2cm}{2cm}

\newtheorem{theorem}{Theorem}[section]
\newtheorem{lemma}[theorem]{Lemma}

\newtheorem{corollary}[theorem]{Corollary}

\theoremstyle{definition}

\newtheorem*{definition*}{Definition}

\theoremstyle{remark}
\newtheorem{remark}[theorem]{Remark}

\numberwithin{equation}{section}

\DeclareMathOperator{\conv}{conv}

\DeclareMathOperator{\inte}{int}

\renewcommand{\epsilon}{\varepsilon}

\begin{document}

\title{Billiards in convex bodies with acute angles}

\author{Arseniy~Akopyan{$^\spadesuit$}}
\thanks{$^\spadesuit$ Supported by People Programme (Marie Curie Actions) of the European Union's Seventh Framework Programme (FP7/2007-2013) under REA grant agreement n$^\circ$[291734].}
\address{{$^\spadesuit$}Arseniy~Akopyan, Institute of Science and Technology Austria (IST Austria), Am Campus~1, 3400 Klosterneuburg, Austria}
\email{akopjan@gmail.com}

\author{Alexey~Balitskiy{$^\clubsuit$}}
\thanks{{$^\clubsuit$} Supported by the Russian Foundation for Basic Research grant 15-31-20403 (mol\_a\_ved).}
\thanks{{$^\clubsuit$} Supported by the Russian Foundation for Basic Research grant 15-01-99563 A}
\thanks{{$^\clubsuit$} Supported in part by the Moebius Contest Foundation for Young Scientists.}
\thanks{{$^\clubsuit$} Supported in part by the Simons Foundation.}
\address{{$^\clubsuit$}Alexey~Balitskiy, Dept. of Mathematics, Moscow Institute of Physics and Technology, Institutskiy per. 9, Dolgoprudny, Russia 141700 \newline
{$^\clubsuit$}Institute for Information Transmission Problems RAS, Bolshoy Karetny per. 19, Moscow, Russia 127994
}
\email{alexey\_m39@mail.ru}

\maketitle

\begin{abstract}
In this paper we investigate the existence of closed billiard trajectories in not necessarily smooth convex bodies. In particular, we show that if a body $K\subset \mathbb{R}^d$ has the property that the tangent cone of every non-smooth point $q\in \partial K$ is acute (in a certain sense) then there is a closed billiard trajectory in $K$.
\end{abstract}

\section{Introduction}
\label{section:introduction}

The problem of existence of closed billiard trajectories in certain domains has a long history (a good reference for a general discussion is~\cite{tabachnikov2005geometry}). It was established that any smooth convex body $K\subset \mathbb{R}^d$ has a closed billiard trajectory with $m$ bounces at the boundary of $K$, for prime $m$ and for some other $m$. For example, some lower bounds for the number of such trajectories in terms of $d$ and $m$ were studied in~\cite{birknoff1960dynamical,farber2002topology2, farber2002topology, karasev2009periodic}.

Another source for substantial current interest of studying billiard trajectories (in more general setting, with the length measured using arbitrary Minkowski norm) is in their relation to symplectic geometry and Hamiltonian dynamics (see~\cite{artstein2014bounds} where the connection is established between billiards and the Hofer--Zehnder symplectic capacity of a Lagrangian product) and classical problems in convexity theory (see~\cite{artstein2014from} where the Mahler conjecture is deduced from the Viterbo conjecture on the volume--capacity inequality using the billiard technique).

In this paper we study the question of existence of closed billiard trajectory in a non-smooth convex body $K$ in $\mathbb{R}^d$. The famous problem of this kind is the widely open problem of existence of a closed billiard trajectory in an obtuse triangle; the strongest result at the moment is the existence of a closed billiard trajectory in triangles with angles not greater then $100^\circ$ (see~\cite{schwartz2009obtuse}).

We have nothing to say about obtuse triangles; instead we mainly consider ``acute-angled'' convex bodies and show that the minimal (by length) ``generalized'' trajectory should be ``classical''. This is the main idea of this paper, though the detail are different in several different theorems presented here.

Let us give some precise definitions. We are going to distinguish between two types of trajectories:

\begin{itemize}
\item Classical trajectories may only have reflection (bounce) points on smooth part of the boundary $\partial K$. At such points the trajectory is reflected as usual, so that the difference of the unit velocities is proportional to the normal.

\item Generalized trajectories may have also reflection points at non-smooth points of $\partial K$. By definition, a reflection $a \rightarrow q \rightarrow b$ of the trajectory traveling from $a \in K$ to $b \in K$ through $q \in \partial K$ is considered to be \emph{generalized billiard} if the bisector of the angle $\widehat{aqb}$ is orthogonal to some support hyperplane of $K$ at the point $q$ (we suppose here that $q$ does not belong to the segment $[a,b]$). In other words, the difference of unit velocities at the reflection point is proportional to \emph{some} outer normal to $K$ at $q$.
\end{itemize}

Now we define the acuteness precisely:

\begin{definition*}
We say that a non-smooth point $q \in \partial K$ satisfies the \emph{acuteness condition} if the tangent cone $T_K(q)$ can be represented as the orthogonal product $T_K(q) = F \times T^{k}$, where $T^k$ is a $k$-dimensional cone with property that for all points $a,b \in T^k$ the inequality $\widehat{aqb} < \pi/2$ holds, and $F$ is an $(d-k)$-dimensional linear subspace orthogonal to $T^k$.
\end{definition*}

\begin{definition*}
If all non-smooth points of $\partial K$ satisfy the above acuteness condition we call $K$ an \emph{acute body}.
\end{definition*}

The main results are the following theorems:

\begin{theorem}
	\label{thm:acute}
	In an acute convex body $K \subset \mathbb{R}^d$ there exists a closed classical billiard trajectory with no more than $d+1$ bounces.
\end{theorem}

The idea of the proof is to show that the shortest closed generalized billiard trajectory do not pass through non-smooth points. In other words, such a trajectory turns out to be always classical. Recall that the shortest closed generalized billiard trajectory always exists and has between $2$ and $d+1$ bounces, this result due to K.~Bezdek and D.~Bezdek is discussed in section~\ref{section:bezdeks}.

\begin{corollary}
\label{corollary:simplex}
In a simplex with all acute dihedral angles (e.g., a simplex close to regular) there exists a closed classical billiard trajectory with $d+1$ bounces.
\end{corollary}

\begin{remark}
\label{remark:acute_simplex}
A simplex with all acute dihedral angles is commonly called \emph{acute} in the literature; see for example~\cite{brandts2007dissection}. Here we use a different definition for acuteness, but it can be seen that for simplices both definitions coincide. Lemma~\ref{lem:acute_simplex_is_acute} establishes this in one direction, and the opposite direction is obvious.
\end{remark}

More generally, we can prove that under some additional conditions on a shortest generalized billiard trajectory the trajectory turns out to be classical.
%
%

\begin{theorem}
\label{theorem:any_body}
If the shortest closed generalized trajectory in $K \subset \mathbb{R}^d$ has precisely $d+1$ bounces, then it is classical.
\end{theorem}

In the last section of the paper we prove generalization of this theorem for the normed space.

\medskip
{\bf Acknowledgments.}
The authors thank Roman Karasev and the unknown referee for their numerous remarks improving the presentation and the language of the paper.

\section{Bezdeks' trajectories}
\label{section:bezdeks}

Let us recall the powerful approach to closed billiard trajectories from~\cite{bezdek2009shortest}. There the problem of finding the length, denoted here by $\xi(K)$, of the shortest closed generalized trajectory in $K$ was restated in terms of minimizing another functional that has a minimum from the compactness considerations.

Let $\ell(Q)$ be the (Euclidean) length of the closed polygonal line $Q$. Using the same notation as in~\cite{akopyan2014elementary} we put
$$
\mathcal P_m(K) = \{(q_1,\ldots, q_m) : \{q_1,\ldots, q_m\} \ \text{doesn't fit into}\ (\inte K + t)\ \text{with} \ t\in \mathbb{R}^d \} =
$$
$$
= \{(q_1,\ldots, q_m) : \{q_1,\ldots, q_m\} \ \text{doesn't fit into}\ (\alpha K+t)\ \text{with}\ \alpha\in (0,1),\ t\in \mathbb{R}^d \}.
$$

Our main tool is:

\begin{theorem}[Theorem 1.1 in~\cite{bezdek2009shortest}]
\label{thm:bezdeks}
For any convex body $K \subset \mathbb{R}^d$ an equality holds:
$$
\xi(K) = \min_{m\ge 2} \min_{Q\in \mathcal P_m(K)} \ell(Q),
$$
and furthermore, the minimum is attained at $m\le d + 1$.
\end{theorem}

\begin{remark}
\label{remark:bezdeks}
Here we need to make an important remark. Suppose a polygonal line $Q$ (that cannot be translated into $\inte K$) has more than $d+1$ vertices and it has no \emph{fake} vertices, that is, no coinciding consecutive vertices and no two consecutive segments in the same direction. Then from the proof in~\cite{bezdek2009shortest} it follows that $\ell(Q)>\xi(K)$.
\end{remark}

\section{Sufficient conditions in the Euclidean case}
\label{section:euclidean}


%
%

\begin{lemma}[Particular case of Lemma 2.2 in~\cite{bezdek2009shortest}]
\label{lemma:bezdek}
Suppose the points $q_1, \ldots, q_m$ satisfy the following condition: There exist affine halfspaces $H_1^+, \ldots, H_m^+$ with outer normals $n_1, \ldots, n_m$, such that
\begin{enumerate}
\item
$q_i \in \partial H_i^+$ for $i=1,\ldots,m$;
\item
$K \subset H_i^+$ for $i=1,\ldots,m$;
\item
$0 \in \conv\{n_1, \ldots, n_m\}$.
\end{enumerate}
Then the polygonal line with vertices $q_1, \ldots, q_m$ (and maybe with some other vertices) cannot be translated into $\inte K$.
\end{lemma}

\begin{proof}
Can be found in~\cite{bezdek2009shortest}.
\end{proof}

\begin{lemma}
\label{lemma:returns}
Suppose a generalized billiard trajectory in $K\subset\mathbb R^d$ with three or more bounces has \emph{a point of return}, that is, a part $q_{i-1} \rightarrow q_i \rightarrow q_{i+1}$ such that $q_{i-1} = q_{i+1}$. Then it cannot be the shortest generalized trajectory.
\end{lemma}

\begin{proof}
Suppose it is the shortest. We note that dropping the point $q_{i-1}$ from the trajectory, we obtain the polygonal line whose length is strictly shorter than it was before and which still cannot be translated into $\inte K$, since it has the same set of vertices. This contradicts Theorem~\ref{thm:bezdeks}.
\end{proof}

We denote by $N_K(q)$ the cone of outer normals and by $T_K(q) = N_K^\circ(q)$ the tangent cone for a point $q \in \partial K$, the latter was already used in the definition of acuteness. We assume that both these cones have the vertices at the origin. If $q$ is non-smooth point of $\partial K$ then $N_K(q)$ is non-trivial (not a single ray).

The proof of Theorem~\ref{thm:acute} will follow from its slightly more general form:

\begin{theorem}
\label{thm:acuteweak}
Suppose that for all non-smooth $q\in\partial K$ and each ray $\rho \subset N_K(q)$ there exists a section $N_K(q)\cap \tau$ by a two-dimensional plane $\tau \supset \rho,$ that contains an angle of measure $>\frac\pi 2$. Then the shortest closed generalized trajectory in $K \subset \mathbb{R}^d$ is classical.
\end{theorem}

\begin{proof}
Assume that the shortest generalized closed trajectory passes through a non-smooth point $q \in \partial K$. The normal cone $N_K(q)$ satisfies the assumption in the theorem.

Let $a \rightarrow q \rightarrow b$ be the part of the trajectory and $n$ be the outer normal at $q$ opposite to the bisector of $\widehat{aqb}$.
Find the plane $\tau$ (containing the ray emanating from $q$ with direction $n$) that cuts from from $q + N_K(q)$ an angle of measure $>\frac\pi 2$.

Denote the vectors of the sides of the angle $\tau \cap q + N_K(q)$ by $n_1$ and $n_2$. Without loss of generality we may assume that $a$ and $n_1$ lie on one side with respect to $n$ and $b$ and $n_2$ lie on the other side.
Denote by $H_1$ and $H_2$ the support hyperplanes at $q$ orthogonal to $n_1$ and $n_2$.
Reflect $a$ and $b$ in hyperplanes $H_1$ and $H_2$ respectively and obtain point $a'$ and $b'$ (see Figure~\ref{fig:reflected points}).

Note that the angle $\widehat{a'qb'}$ is less than $\pi$, since the points $a'$ and $b'$ lie in the open halfspace bounded by the hyperplane through $q$, whose normal is the reflection of $n$ in the bisector hyperplane of $H_1$ and $H_2$.
Let $q_1$ and $q_2$ be the points of intersection of the segment $[a',b']$ with $H_1$ and $H_2$.
Then
\begin{equation*}
	|aq_1|+|q_1q_2|+|q_2b|=|a'q_1|+|q_1q_2|+|q_2b'|=|a'b'|<|a'q|+|qb'|=|aq|+|qb|.
\end{equation*}

Thus if we replace $a \rightarrow q \rightarrow b$ with $a \rightarrow q_1 \rightarrow q_2 \rightarrow b$ then the trajectory becomes shorter, but the normals at the vertices of the trajectory still surround the origin, because $n$ is positive combination of $n_1$ and $n_2$. This certifies that the new trajectory still cannot be translated into $\inte K$.

\begin{figure}[h]
\centering
\includegraphics{cutefigures-3.mps}
\caption \,
\label{fig:reflected points}
\end{figure}

\end{proof}

Let us show that Theorem~\ref{thm:acuteweak} is indeed more general than Theorem~\ref{thm:acute}.
\begin{lemma}
	\label{lem:acuteweak is not weak}
	Acute bodies satisfy the assumption of Theorem~\ref{thm:acuteweak}.
\end{lemma}

\begin{proof}
Suppose a non-smooth point $q \in \partial K$ satisfies the acuteness condition.
Let $T_K(q) = F \times T^{k}$ be the orthogonal decomposition from the acuteness definition, where $T^k$ is $k$-dimensional acute cone, whose diameter equals $\varphi < \frac \pi 2$. Then the cone of outer normals $N_K(q) = T_K^\circ(q)$ is $k$-dimensional. 
Denote by $L$ the $k$-dimensional linear hull of $N_K(q)$.

Consider the ray $\rho \subset N_K(q)$ from the origin and denote by $p$ an arbitrary point (different from $q$) on the ray from $q$ in the direction opposite to $\rho$.
Let $\psi = \sup\limits_{b \in K \cap (q+L)} \widehat{pqb}$. From the acuteness $\psi < \frac \pi 2$. Let $b$ be such a point that $\widehat{pqb} = \psi - \varepsilon$ for $0 < \varepsilon < \frac \pi 2 - \varphi$.

\begin{figure}[h]
\centering
\includegraphics{cutefigures-2.mps}
\caption \,
\label{pic2}
\end{figure}

Now we work in the two-dimensional plane $\tau$ through the points $p,q,b$. In this plane we draw the line $\ell_1$ through $q$ forming angle $\psi$ with $(qp)$, and line $\ell_2$ through $q$ forming angle $\varphi$ with $(qb)$ (see Figure~\ref{pic2}). Note that the hyperplanes $H_1$ and $H_2$, passing through $\ell_1, \ell_2$ and orthogonal to $\tau$, are support hyperplanes for $K$ ($H_1$ is such because of the definition of $\psi$, and $H_2$ is such because of the acuteness condition). This is exactly the construction needed in Theorem~\ref{thm:acuteweak}: The section of $N_K(q)$ by $\tau - q \supset \rho$ contains an angle of measure $>\frac\pi 2$.
\end{proof}

\begin{proof}[Proof of Corollary~\ref{corollary:simplex}]
In Lemma~\ref{lem:acute_simplex_is_acute} below we show that such simplex is indeed acute in the sense of our definition of acuteness. Now consider the shortest generalized trajectory and show that it is classical and has $d+1$ bounces. The first conclusion follows from Theorem~\ref{thm:acute}. The second conclusion follows form the observation that the outer normals to some $d$ facets of the simplex cannot surround the origin.
\end{proof}

\begin{lemma}
\label{lem:acute_simplex_is_acute}
A simplex with all acute dihedral angles satisfies the acuteness condition.
\end{lemma}

\begin{proof}
Consider a simplex $S = \conv\{v_0, v_1, \ldots, v_d\} \subset \mathbb{R}^d$ having all acute dihedral angles. It is known that any face of such a simplex also has only acute dihedral angles (see~\cite[Satz~4]{fiedler1957qualitative}~or~\cite[Proposition~2.7]{brandts2007dissection}).

Now consider a non-smooth point $q$ belonging to, say, $k$-dimensional ($0 \le k \le d-2$) face $F = \conv\{v_0, \ldots, v_k\}$. The corresponding tangent cone decomposes orthogonally as $T_S(q) = F \times C$, where $C$ is a simplicial cone ($(d-k)$-hedral angle) in some $(d-k)$-dimensional subspace $L$, generated by its extremal rays $\rho_{k+1}, \ldots, \rho_d$; let the ray $\rho_i$ be parallel to the face $\conv (F\cup\{v_i\})$ and be orthogonal to the face $F$.

Note that the angle between $\rho_i$ and $\rho_j$ in $L$ equals a certain dihedral angle of the $(k+2)$-dimensional face $\conv (F\cup\{v_i,v_j\})$, so this angle is acute. Consider the $(d-k)$-dimensional cone $C$ as a $(d-k-1)$-dimensional spherical simplex in $\mathbb S^{d-k-1}$. All its edges are less than $\frac \pi 2$, thus a  simple convexity argument implies that its diameter is attained at some its edge and is less than $\frac \pi 2$. Therefore, the acuteness condition is fulfilled.
\end{proof}

\begin{remark}
\label{remark:different_corollary_proof}
Corollary~\ref{corollary:simplex} can be proved without using Fiedler's theorem~\cite[Satz~4]{fiedler1957qualitative}. It can be directly shown that a simplex with only acute dihedral angles satisfies the condition of Theorem~\ref{thm:acuteweak} by considering what the normal fan of the simplex cuts on the sphere $\mathbb S^{d-1}$.
\end{remark}

In a certain particular case the acuteness condition may be omitted. The corresponding result is Theorem~\ref{theorem:any_body}, which we are going to prove now:

\begin{proof}[Proof of Theorem~\ref{theorem:any_body}]
Assume $q_1, \ldots, q_{d+1}$ form the shortest closed generalized trajectory in $K \subset \mathbb{R}^d$.

Consider the outer normals to support hyperplanes $H_1, \ldots, H_{d+1}$ at the points $q_1, \ldots, q_{d+1}$. They are opposite to the bisectors of $\widehat{q_{i-1}q_iq_{i+1}}$: $n_i = \frac{q_i - q_{i-1}}{|q_i - q_{i-1}|} - \frac{q_{i+1} - q_i}{|q_{i+1} - q_i|}$ (the indexing is cyclic).

Note that a positive combination of $n_i$ is zero. First, consider the case when $n_i$ span a proper hyperspace of $\mathbb{R}^d$. Then one of $n_i$'s can be dropped keeping the condition $0\in\conv\{n_i\}$ (this follows from the Carath\'eodory theorem). We also omit the corresponding $q_i$ from the trajectory. Then, the obtained polygonal line becomes strictly shorter but still cannot be translated into $\inte K$ (see Remark~\ref{remark:bezdeks}). Thus it remains to consider the case $0\in\inte\conv\{n_i\}$ (here we essentially use the assumption that the trajectory has the maximal number $d+1$ of bounces).


Second, we note that the trajectory does not have points of return (see Lemma~\ref{lemma:returns}).

Now assume that $a \rightarrow q \rightarrow b$ is a fragment of the shortest generalized trajectory near the non-smooth point $q \in \partial K$, and $a,q,b$ do not lie on the same line. Consider the cone $N_K(q)$ of outer normals. Since it is non-trivial, it contains rays from the origin other than the ray $\rho$ that is opposite to the bisector of $\widehat{aqb}$. We rotate $\rho$ slightly to obtain $\tilde{\rho} \in N_K(q), \tilde{\rho} \neq \rho$. Consider the support hyperplane $H$ at $q$ with outer normal $\tilde{\rho}$.


As shown in Lemma~\ref{lemma:minimize} below, the point $q$ can be shifted along $H$ so that the length $|a-q|+|q-b|$ becomes strictly smaller. It remains to show that, if we replace $a \rightarrow q \rightarrow b$ with $a \rightarrow \tilde{q} \rightarrow b$ in the trajectory, then the trajectory still cannot be translated into $\inte K$. After the replacement, one of the support hyperplanes at the vertices slightly rotates (its outer normal $\rho$ is replaced with $\tilde{\rho}$). Since the rotation is slight we can assume that the origin still belongs to the convex hull of normals to the support hyperplanes (recall that $0\in\inte\conv\{n_i\}$). Thus, Lemma~\ref{lemma:bezdek} yields the required statement.
\end{proof}

\begin{lemma}
\label{lemma:minimize}
Let $a,q,b$ be points in $\mathbb{R}^d$ not lying on the same line, $H$ be a hyperplane such that $q\in H$ and $a,b$ lie at the same closed halfspace bounded by $H$. Suppose that $H$ is not orthogonal to the bisector of $\widehat{aqb}$. Then $q$ does not deliver the minimum of $|a-r|+|r-b|$ subject to $r \in H$.
\end{lemma}

\begin{proof}
Reflect the point $b$ in $H$ and obtain the point $b'$. By the assumptions of the lemma, $a,q,b'$ do not lie on the same line, so $|a-q|+|q-b|$ can be made smaller if we shift $q$ to a point from $[a,b'] \cap H$.
\end{proof}

\section{Sufficient conditions in arbitrary normed spaces}
\label{section:non-euclidean}

Let us extend the proof of Theorem~\ref{theorem:any_body} to the case of the generalized reflection law in a normed space.

Let an $d$-dimensional real vector space $V = \mathbb{R}^d$ be endowed with a norm with unit ball $T^\circ$ (where $T^\circ \subset V$ is polar to a convex body $T \subset V^*$). We follow the notation of~\cite{akopyan2014elementary} and denote such a norm by $\|\cdot\|_T$.

By definition, $\|q\|_T = \max\limits_{p \in T} \langle p, q \rangle$, where $\langle \cdot, \cdot \rangle : V^* \times V \rightarrow \mathbb{R}$ is the canonical bilinear form of the duality between $V$ and $V^*$. Here we assume that $T$ contains the origin (although this can be relaxed to some extent), but is not necessarily centrally symmetric. Therefore the norm may be non-symmetric, in general, $\|q\|_T \neq \|-q\|_T$.
In what follows, we always assume that $T$ is smooth and thus $T^\circ$ is strictly convex.

We measure lengths in $V$ using the norm $\|\cdot\|_{T}$ and the billiard reflection rule is given by locally minimizing the length functional. We say that a polygonal line $q_{start} \rightarrow q_{refl} \rightarrow q_{end}$ (where $q_{refl} \in \partial K, q_{start} \in K, q_{end} \in K$) has a \emph{billiard reflection} at the point $q_{refl}$ if there exists a support hyperplane $H$ for the body $K$ at the point $q_{refl}$ such that the functional
$$
\varphi(q) = \|q_{end} - q\|_T + \|q - q_{start}\|_T
$$
has a local minimum at the point $q = q_{refl}$ under the constraint $q \in H$. If $q_{refl}$ belongs to the smooth piece of $\partial K$ we say that a \emph{classical} billiard reflection occurs.
In such a case one can rewrite the reflection rule in the differential form:
\begin{equation}
\label{equation:reflection}
p' - p = - \lambda n_K(q), \quad \lambda > 0.
\end{equation}
Here we define the momenta $p, p' \in \partial T \subset V^*$ before and after the reflection so that $p$ is a functional reaching its maximum at $q_{end} - q$, and $p'$ is a functional reaching its maximum at $q - q_{start}$ (if $T$ is strictly convex then such $p$ and $p'$ are uniquely defined).


Also here we define the outer normal to the body $K$ at a point $q \in \partial K$ as
$$
n_K(q) = d\|q\|_{K^\circ}, \quad n_K(q) \in \partial K^\circ.
$$
The cone $N_K(q)$ of outer normals is defined by
$$
N_K(q) = \{n \in V^*: \langle n, q' - q\rangle \le 0 \ \forall q' \in K\}.
$$
It can be easily checked that in the case of a smooth point $q \in \partial K$ the above definitions of normals agree: $N_K(q) = \{n_K(q)\}$.

Thus the \emph{generalized reflection law} looks like
\begin{equation}
\label{equation:generalizedreflection}
p' - p \in -N_K(q).
\end{equation}

We again use the notions
$$
\mathcal P_m(K) = \{(q_1,\ldots, q_m) : \{q_1,\ldots, q_m\} \ \text{doesn't fit into}\ (\inte K + t)\ \text{with} \ t\in V \} =
$$
$$
= \{(q_1,\ldots, q_m) : \{q_1,\ldots, q_m\} \ \text{doesn't fit into}\ (\alpha K+t)\ \text{with}\ \alpha\in (0,1),\ t\in V \}
$$
and
$$
\xi_T(K) = \min_{Q \in \mathcal Q_T(K)} \ell_T(Q),
$$
where $Q = (q_1,\ldots, q_m)$, $\ m \ge 2,$ ranges over the set $\mathcal Q_T(K)$ of all closed generalized billiard trajectories in $K$ with geometry defined by $T$. (Here we denote the length $\ell_T (q_1,\ldots, q_m) = \sum_{i=1}^m \|q_{i+1} - q_i\|_T$.)

The generalization of the main result of~\cite{bezdek2009shortest}, proved in~\cite{akopyan2014elementary}, is the following:

\begin{theorem}
\label{thm:bezdeksrevisited}
For any convex bodies $K \subset V$, $T \subset V^*$ ($T$ is smooth) containing the origins of $V$ and $V^*$ in their interiors the equality holds:
$$
\xi_T(K) = \min_{m\ge 2} \min_{Q\in \mathcal P_m(K)} \ell_T(Q);
$$
and furthermore, the minimum is attained at $m\le d + 1$.
\end{theorem}

\begin{remark}
\label{remark:bezdeksrevisited}
Lemma \ref{lemma:bezdek} still holds in this setting (see~\cite{akopyan2014elementary}). It is used in the proof of~\ref{thm:bezdeksrevisited}.
\end{remark}

\begin{remark}
\label{remark:xidef}
Actually, Theorem~\ref{thm:bezdeksrevisited} is proved in~\cite{akopyan2014elementary} only for smooth body $K$ and classical trajectories. But approximating non-smooth $K$ in Hausdorff metrics and passing to the limit we obtain the formulation above. Note that the formula of Theorem~\ref{thm:bezdeksrevisited} can be used as the definition of $\xi_T(K)$ for arbitrary $T$ and $K$ without any smoothness assumptions.
\end{remark}

To proceed further, we generalize Lemma~\ref{lemma:minimize}:

\begin{lemma}
\label{lemma:finslerminimize}
Let $a,q,b$ be points in $\mathbb{R}^d$ not lying on the same line, $H$ be a hyperplane with normal $n \in V^*$ such that $q\in H$ and $a,b$ lie at the same closed halfspace bounded by $H$. Suppose the length is measured using the norm with unit body $T^\circ$, such that $T$ is strictly convex.

Let $p,p' \in \partial T$ be the uniquely defined momenta of the segments $a \rightarrow q, q \rightarrow b$. Suppose that $\langle n, p' - p \rangle \neq 0$. Then $q$ does not deliver the minimum of $\|a-r\|_T+\|r-b\|_T$ subject to $r \in H$.
\end{lemma}

\begin{proof}
If $q$ delivered the minimum of $\|a-r\|_T+\|r-b\|_T$ subject to $r \in H$ then the reflection law~\ref{equation:reflection} would imply that $\langle n, p' - p \rangle = 0$.
\end{proof}

\begin{lemma}
\label{lemma:triangleinequality}
Let $a,q,b$ be such that $q$ does not belong to the segment $[a,b]$. Suppose the length is measured using the norm with unit body $T^\circ$. Then $\|a-q\|_T+\|q-b\|_T > \|a-b\|_T$ subject to $r \in H$.
\end{lemma}

\begin{proof}
The result follows easily from the strict convexity of $T^\circ$ (which follows from the smoothness of $T$).
\end{proof}

\begin{remark}
\label{remark:bezdeksnoneuclid}
As in the Euclidean case, if a polygonal line $Q$ (that cannot be translated into $\inte K$) has more than $d+1$ vertices and it has no \emph{fake} vertices, then its length is strictly greater than $\xi_T(K)$. The argument is the same as in Remark~\ref{remark:bezdeks}, but makes use of Lemma~\ref{lemma:triangleinequality}.
\end{remark}

\begin{remark}
\label{remark:returnsnoneuclid}
Lemma \ref{lemma:triangleinequality} and Remark~\ref{remark:bezdeksnoneuclid} allow us to prove Lemma~\ref{lemma:returns} in the non-Euclidean case.
\end{remark}

Here comes our final result:

\begin{theorem}
\label{theorem:noneuclid}
Suppose the length is measured using the norm with strictly convex unit body $T^\circ$ such that $T$ is strictly convex too (in other words, $T$ is smooth and strictly convex).

If the shortest closed generalized trajectory in $K \subset \mathbb{R}^d$ has $d+1$ bounces, then it is classical, that is, it does not pass through non-smooth points of $\partial K$.
\end{theorem}

\begin{proof}
Assume that $q_1, \ldots, q_{d+1}$ form the shortest closed generalized trajectory in $K \subset \mathbb{R}^d$.
Denote by $p_i$ the momentum along the trajectory segment $q_{i-1} \rightarrow q_i$ (the indexing is cyclic). Consider support hyperplanes $H_1, \ldots, H_{d+1}$ to $K$ at the points $q_1, \ldots, q_{d+1}$ with normals $p_2-p_1, p_3-p_2, \ldots, p_1-p_{d+1}$ respectively.

There are two possibilities: Either $p_2-p_1, \ldots, p_1-p_{d+1}$ with zero sum span all $V^*$ or they are contained in a hyperplane. The latter is impossible, and the argument is the same as in the proof of Theorem~\ref{theorem:any_body}. In this case we drop one of the $q_i$ keeping the condition $0\in\inte \conv \{p_{i+1}-p_i\}$; the obtained polygonal line is strictly shorter but still cannot be translated into $\inte K$ (see Remark~\ref{remark:bezdeksrevisited}). So we can consider that $0 \in \inte \conv \{p_2-p_1, \ldots, p_1-p_{d+1}\}$.

Next, we note that the trajectory does not have points of return (see Remark~\ref{remark:returnsnoneuclid}).

Then, assume that $a \rightarrow q \rightarrow b$ is a fragment of the shortest closed generalized trajectory near the non-smooth point $q \in \partial K$, and $a,q,b$ do not lie on the same line. Consider the cone $N_K(q)$ of outer normals. Since it is non-trivial, it contains rays from the origin other than the ray $\rho$ which is opposite to $p'-p$, where $p$ and $p'$ are the momenta of the trajectory parts $a \rightarrow q$ and $q \rightarrow b$. Let us rotate $\rho$ slightly to obtain $\tilde{\rho} \in N_K(q), \tilde{\rho} \neq \rho$. Consider the support hyperplane $H$ at $q$ with outer normal $\tilde{\rho}$.

As shown in Lemma~\ref{lemma:finslerminimize}, the point $q$ can be shifted along $H$ so the length $\|a-q\|_T+\|q-b\|_T$ becomes strictly smaller. It remains to show that, if we replace $a \rightarrow q \rightarrow b$ with $a \rightarrow \tilde{q} \rightarrow b$ in the trajectory, then it still cannot be translated into $\inte K$. After the replacement, one of the support hyperplanes at the vertices slightly rotates (its outer normal $\rho$ is replaced with $\tilde{\rho}$). Since the rotation is slight we can assume that the origin still belongs to convex hull of the normals to those support hyperplanes. Thus, the non-Euclidean version of Lemma~\ref{lemma:bezdek} yields the required conclusion.
\end{proof}

\bibliographystyle{abbrv}
\bibliography{billiards_in_acute_angles}{}
\end{document}